\documentclass[a4paper,reqno,12pt]{amsart}
\usepackage{verbatim}
\usepackage{amssymb,amsmath,amsthm}
\usepackage{amsfonts}
\usepackage{array}

\newtheorem{theorem}{Theorem}
\newtheorem{lemma}{Lemma}

\theoremstyle{remark}
\newtheorem{remark}[theorem]{\bf Remark}

\newcommand{\leg}[2]{\left(\frac{#1}{#2}\right)}

\begin{document}

\setcounter{page}{1}

\title[]{Two Triple binomial sum supercongruences}

\author{Tewodros~Amdeberhan  and Roberto~Tauraso}

\address{Department of Mathematics,
Tulane University, New Orleans, LA 70118, USA}
\email{tamdeber@tulane.edu}

\address{Dipartimento di Matematica, 
Universit\`a di Roma ``Tor Vergata'', 
via della Ricerca Scientifica,
00133 Roma, Italy}
\email{tauraso@mat.uniroma2.it}


\date{\today}

\begin{abstract} In a recent article, Apagodu and Zeilberger discuss some applications of an algorithm for finding and proving congruence identities (modulo primes)
of indefinite sums of many combinatorial sequence. At the end, they propose some supercongruences as conjectures. Here we prove one of them, including a new companion enumerating abelian squares, and we leave some remarks for the others. 
\end{abstract}

\maketitle

\section{Introduction}

\noindent
In recent literature, a variety of supercongruences have been conjectured by several people, such as
Beukers \cite{B}, van Hamme \cite{H}, Rodriguez-Villegas \cite{RV}, Zudilin \cite{Z}, Chan et al.  \cite{CCS}, 
and lots more by Z.-W. Sun \cite{ZWS1}, \cite{ZWS2}.  Some of these conjectures are proved using a variety of methods, including the Gaussian hypergeometric series, the Wilf-Zeilberger method and $p$-adic analysis. 

\smallskip
\noindent
In a more recent and illuminating article, Apagodu and Zeilberger \cite{AZ} discuss applications of \emph{constant term extraction} for finding and proving congruence identities (modulo primes) of indefinite sums of many combinatorial sequence. At the end of their paper, they propose some supercongruences for multiple sums as conjectures. In this paper, we prove one of them, 
including a new companion, and we leave some remarks for the others. To be specific, we confirm the following supercongruence that appears as Conjecture 6' in \cite{AZ}.  
\begin{theorem} \label{Main1}
Let $p>2$ be a prime, and let $r$, $s$, $t$ be any positive integers, then
\begin{equation}\label{SS}
\sum_{m_1=0}^{rp-1}\sum_{m_2=0}^{sp-1}\sum_{m_3=0}^{tp-1}
\binom{m_1+m_2+m_3}{m_1,m_2,m_3}\equiv_{p^3} 
\sum_{m_1=0}^{r-1}\sum_{m_2=0}^{s-1}\sum_{m_3=0}^{t-1}
\binom{m_1+m_2+m_3}{m_1,m_2,m_3}.
\end{equation}
\end{theorem}

\smallskip
\noindent
The paper is organized as follows. The preliminary sections, Section 2 and Section 3, target a host of results which are relevant for our present purpose. In Section 4, we supply the proof for one of our main results - Theorem \ref{Main1}. The penultimate section, Section 5, recalls the notion of \emph{abelian squares} \cite{R} and subsequently provides congruences for a triple sum of squares of multinomial terms. The final section, Section 6, concludes with a few remarks. 

\section{Preliminary results on harmonic sums}

\noindent  For $r>0$, let
${\bf s}=(s_1, \ldots, s_r)\in (\mathbb{Z^{\ast}})^r$ and let $x\in\mathbb{R}$. We define the multiple sum
$$
H_n({\bf s};x)=
\sum_{1\le k_1<\cdots<k_r\le n}
\prod_{i=1}^r\frac{x_i^{k_i}}{k_i^{|s_i|}}
\quad \mbox{with}\quad x_i = \left\{\begin{array}{lr}
        x & \mbox{if $s_i<0$, }\\
        1 & \mbox{if $s_i>0$. }
        \end{array}\right.$$
The number $l({\bf s}):=r$ is called the depth (or length) and $|{\bf s}|:=\sum_{j=1}^r|s_j|$ is the weight of the multiple sum.
By convention, these sums are zero if $n<r$. $H_n({\bf s};1)$ is the {\sl ordinary multiple harmonic sum} and in that case we will simply write $H_n({\bf s})$. We denote the {\sl ordinary harmonic sum} $H_n(1)$ by $H_n$. Some known harmonic-sum evaluations include:
\begin{align}
\label{H1}
&\sum_{k=0}^{n} \binom{n}{k}H_k=2^n\left(H_n-\sum_{k=1}^{n}\frac{1}{k2^k}\right)\;,\\
\label{H2}
&\sum_{k=0}^{n}\binom{n}{k}^2H_k=\binom{2n}{n}(2H_n-H_{2n})\;.
\end{align}
These identities \eqref{H1} and \eqref{H2} are found in  \cite[(39)]{S} and \cite[(3.125)]{G}, respectively.

\smallskip
\noindent 
Let $T(n,k)=(-1)^{n-k}\binom{n}{k}\binom{n+k}{k}$. Then, we have
\begin{align}\label{T1}
&\sum_{k=0}^nT(n,k)=1\,,\\
\label{T2}&\sum_{k=0}^nT(n,k)H_k=2H_n\,.
\end{align}
The identity \eqref{T1} is found in 
\cite[(3.150)]{G} and the identity \eqref{T2} in \cite[Theorem 2]{PR}.

\begin{lemma} The following holds:
\begin{equation} \label{T4} 
\sum_{k=0}^nT(n,k)H_{2k}=3H_n-H_{\lfloor n/2\rfloor}. 
\end{equation}
\end{lemma}
\begin{proof} Let $F(n,k)=\frac{(-1)^{n-k}}{2k+1}
\binom{n-1}k\binom{n+k}kn$ and $\delta_E(n)$ be the \emph{indicator function} of $E\subset\mathbb{Z}$. We first show that
\begin{equation}\label{T4a} 
\sum_{k=0}^{n-1}F(n,k)=-\delta_{odd}(n),
\end{equation}
We follow the Wilf-Zeilberger method. Define $W(n,k)=2(-1)^{n-k-1}\binom{n}{k-1}\binom{n+k}{k-1}$ and then (routinely) check that $F(n+2,k)-F(n,k)=W(n,k+1)-W(n,k)$. Summing over all integers $k$ and telescoping, we get $\sum_{k=0}^{n+1}F(n+2,k)=\sum_{k=0}^{n-1}F(n,k)$. For initial conditions, compute at $n=1$ and $n=2$. This settles the argument on this parity result (\ref{T4a}).

\smallskip
\noindent Now we prove, by induction on $n$, that
\begin{equation}\label{T4b} 
\sum_{j=0}^{n-1}\frac1{2j+1}\sum_{k=0}^jT(n,k)=H_{2n}-
\frac{5}{2}H_n+H_{\lfloor n/2\rfloor}.
\end{equation}
The case $n=1$ checks $-1=-1$. Assume the statement holds for $n-1$. Define $G(n,k)=2(-1)^{n-k}\binom{n}{k-1}\binom{n+k}{k-1}$ and verify that
$T(n,k)-T(n-1,k)=G(n-1,k+1)-G(n-1,k)$, also the summation $\sum_{k=0}^jT(n,k)=\sum_{k=0}^jT(n-1,k)+G(n-1,j+1)$. 
Based on \eqref{T1}, the induction assumption and the identity (\ref{T4a}) from above, we gather
\begin{align*}
\sum_{j=0}^{n-1}\sum_{k=0}^j\frac{T(n,k)}{2j+1}&=\sum_{a=0}^{n-1}\sum_{k=0}^j\frac{T(n-1,k)}{2j+1}
+2\sum_{j=0}^{n-1}\frac{(-1)^{n-j}}{2j+1}\binom{n-1}j\binom{n+j}j \\
&=H_{2n-2}-\frac{5H_{n-1}}2+H_{\lfloor(n-1)/2\rfloor}+\sum_{k=0}^{n-1}\frac{T(n-1,k)}{2n-1}-\frac{2\delta_{odd}(n)}n \\
&=H_{2n-2}-\frac{5H_{n-1}}2+H_{\lfloor(n-1)/2\rfloor}+\frac1{2n-1}-\frac{2\delta_{odd}(n)}n \\
&=H_{2n}-\frac52H_n+H_{\lfloor(n-1)/2\rfloor}+\frac{2\delta_{even}(n)}n
=H_{2n}-\frac{5}{2}H_n+H_{\lfloor n/2\rfloor}. \end{align*}
By induction, the claim (\ref{T4b}) holds true.

\smallskip
\noindent
To prove \eqref{T4}: employ $H_{2k}=\sum_{j=1}^k\frac1{2j-1}+\frac12H_k$ twice, Abel's summation by parts, equations (\ref{T2}) and (\ref{T1}), again $\sum_{j=1}^n\frac1{2j-1}=H_{2n}-\frac12H_n$, and then equation (\ref{T4b}). The outcome is 
\begin{align*} 
\sum_{k=0}^nT(n,k)H_{2k}&=
\sum_{k=0}^nT(n,k)\sum_{j=1}^k\frac1{2j-1}
+\frac{1}{2}\sum_{k=0}^nT(n,k)H_{k}\\
&=\sum_{k=0}^nT(n,k)\sum_{j=1}^n\frac1{2j-1}
-\sum_{j=0}^{n-1}\frac1{2j+1}\sum_{k=0}^jT(n,k)+H_n \\
&=H_{2n}-\frac12H_n-H_{2n}+\frac52H_n-H_{\lfloor n/2\rfloor}+H_n  \\
&=3H_n-H_{\lfloor n/2\rfloor}. \end{align*}
The proof is now complete.
\end{proof}

\noindent
Let $q_p(2)=(2^{p-1}-1)/p$ and $p>3$ a prime. These are known (see \cite[Sections 1 and 7]{TZ}): 
\begin{align}
&H_{p-1}\equiv_{p^2} 0\;,\label{C1}\\
&H_{p-1}(2)\equiv_{p} 0\;,\label{C2}\\
&H_{p-1}(1,1)\equiv_{p} 0\;,\label{C3}\\
&H_{p-1}(-1;2)\equiv_{p^2} -2q_p(2)\;,\label{C4}\\
&H_{p-1}(-1;1/2)\equiv_{p} q_p(2)\;,\label{C44}\\
&H_{p-1}(-2;-1)\equiv_{p} 0\;,\label{C66}\\
&H_{p-1}(-2;2)\equiv_{p} -q^2_p(2)\;,\label{C10}\\
&H_{p-1}(1,-1;-1)\equiv_{p} q^2_p(2)\;,\label{C55}\\
&H_{p-1}(1,-1;2)\equiv_{p} 0\;,\label{C5}\\
&H_{p-1}(-1,1;1/2)\equiv_{p} 0\;. \label{C6}
\end{align}

\section{More preliminary results}

\noindent
In the next section we will need the following results.

\begin{lemma} 
Let $p>2$ be a prime, then we have
\begin{equation}\label{B1}
\sum_{k=1}^{p-1}\frac{1}{k2^k}\sum_{j=1}^{k-1} \frac{2^j}{j}\equiv_{p} 0,
\end{equation}
and
\begin{equation}\label{B2}
\sum_{k=1}^{p-1}\frac{2^k}{k}\sum_{j=1}^{k-1} \frac{1}{j2^j}\equiv_{p}
-2q^2_p(2).
\end{equation}
\end{lemma} 
\begin{proof}
\noindent By \eqref{C6},
\begin{align*}
\sum_{k=1}^{p-1}\frac{1}{k2^k}\sum_{j=1}^{k-1} \frac{2^j}{j}
&\equiv_{p} \sum_{k=1}^{p-1}\frac{2^{-k}}{k}\sum_{j=1}^{k-1} \frac{2^{k-j}}{k-j}
=\sum_{k=1}^{p-1}\sum_{j=1}^{k-1} \frac{2^{-j}}{k(k-j)}\\
&=\sum_{j=1}^{p-2}\frac{2^{-j}}{j}\sum_{k=j+1}^{p-1}\left(\frac{1}{k-j}-\frac{1}{k}\right)
=\sum_{j=1}^{p-2}\frac{2^{-j}}{j}\sum_{k=1}^{p-j-1}\frac{1}{k}-H_{p-1}(-1,1;1/2)\\
&=\sum_{j=1}^{p-2}\frac{2^{-j}}{j}\sum_{k=j+1}^{p-1}\frac{1}{p-k}-H_{p-1}(-1,1;1/2)
\equiv_{p}-2H_{p-1}(-1,1;1/2)\equiv_{p}0.
\end{align*}
As regards \eqref{B2}, use \eqref{C2}, \eqref{C4}, \eqref{C44}, and \eqref{B1}. It follows from 
$$H_{p-1}(-1;2)\cdot H_{p-1}(-1;1/2)=H_{p-1}(2)+
\sum_{k=1}^{p-1}\frac{1}{k2^k}\sum_{j=1}^{k-1} \frac{2^j}{j}
+\sum_{k=1}^{p-1}\frac{2^k}{k}\sum_{j=1}^{k-1} \frac{1}{j2^j}.$$
\end{proof}

\begin{lemma} 
Let $p>2$ be a prime, then 
\begin{equation}\label{C7}
\sum_{k=2}^{p-1}\frac{1}{k^2}\sum_{m=k}^{p-1} \frac{(-1)^{m-k}}{\binom{m}{k}}\equiv_{p}
-2q_p(2).
\end{equation}
\end{lemma}
\begin{proof} We have that
\begin{align*}
\sum_{k=2}^{p-1}\frac{1}{k^2}\sum_{m=k}^{p-1} \frac{(-1)^{m-k}}{\binom{m}{k}}
&=\sum_{k=2}^{p-1}\frac{1}{k^2}\sum_{m=0}^{p-1-k} \frac{(-1)^{m}}{\binom{k+m}{m}}
=\sum_{k=1}^{p-2}\frac{1}{(p-k)^2}\sum_{m=0}^{k-1} \frac{(-1)^{m}}{\binom{p-k+m}{m}}\\
&\equiv_{p} \sum_{k=1}^{p-2}\frac{1}{k^2}\sum_{m=0}^{k-1} \frac{1}{\binom{k-1}{m}}
=\sum_{k=1}^{p-2}\frac{1}{k2^k}\sum_{j=1}^{k} \frac{2^j}{j}\\
&=\sum_{k=1}^{p-1}\frac{1}{k2^k}\sum_{j=1}^{k-1} \frac{2^j}{j}
+H_{p-1}(2)-\frac{H_{p-1}(-1;2)}{(p-1)2^{p-1}}\\
&\equiv_{p} 0+0 -2q_p(2)=-2q_p(2)
\end{align*}
where we have used the congruences \eqref{C4}, \eqref{B1},
$$(-1)^m\binom{p-k+m}{m}=\frac{1}{m!}\prod_{j=k-m}^{k-1}(p-j)\equiv_{p}\binom{k-1}{m},$$
and the identity 
$$\sum_{m=0}^{k-1} \frac{1}{\binom{k-1}{m}}=\frac{k}{2^k}\sum_{j=1}^{k} \frac{2^j}{j},$$
for which the reader may refer to \cite[(2.4)]{G}.
\end{proof}

\begin{lemma} 
Suppose  $i$ and $j$ are non-negative integers and $p>2$ is a prime. Then, for $0< r<p$, we have
\begin{equation}\label{CC22}
\binom{(i+j)p}{r+ip}\equiv_{p^3}
\binom{i+j}{i}\binom{p}{r}j\left(1-p\left((i+j-1)H_{r-1}+\frac{i}{r}\right)\right),
\end{equation}
and
\begin{equation}\label{CC2}
\sum_{m=0}^{p-1}\binom{p-1+(i+j)p}{m+ip}\equiv_{p^3}
\binom{i+j}{i}\left(1+(i+j+1)pq_p(2)+\binom{i+j+1}{2} p^2q^2_p(2)\right).
\end{equation}
\end{lemma}
\begin{proof} We have that
\begin{align*}
\binom{(i+j)p}{r+ip}&=
\binom{(i+j)p}{ip}\frac{jp}{ip+r}
\prod_{k=1}^{r-1}\frac{jp-k}{ip+k}\\
&\equiv_{p^3}\binom{i+j}{i}\frac{jp(-1)^{j-1}}{ip+r}
\prod_{k=1}^{r-1}\frac{1-\frac{jp}{k}}{1+\frac{ip}{k}}\\
&\equiv_{p^3}\binom{i+j}{i}\frac{jp(-1)^{j-1}}{r}\left(1-\frac{ip}{r}\right)
\left(1-jpH_{r-1}\right)\left(1-ipH_{r-1}\right)\\
&\equiv_{p^3}\binom{i+j}{i}\frac{jp(-1)^{r-1}}{r}
\left(1-p\left((i+j)H_{r-1}+\frac{i}{r}\right)\right).
\end{align*}
Congruence \eqref{CC22} follows as soon as we note that
$$\binom{p}{r}\equiv_{p^3}\frac{p(-1)^{r-1}}{r}
\left(1-pH_{r-1}\right).$$
As regards \eqref{CC2},
\begin{align*}
\sum_{m=0}^{p-1}\binom{p-1+(i+j)p}{m+ip}&=
\sum_{m=0}^{p-1}\sum_{l=0}^{p-1}\binom{(i+j)p}{m-l+ip}\binom{p-1}{l}\\
&=\binom{(i+j)p}{ip}2^{p-1}+\sum_{l=1}^{p-1}\binom{p-1}{l}
\sum_{r=1}^{l}\left(\binom{(i+j)p}{r+ip}+\binom{(i+j)p}{r+jp}\right)\\
&\equiv_{p^3} \binom{i+j}{i}2^{p-1}
+\binom{i+j}{i}\sum_{l=1}^{p-1}\binom{p-1}{l}\sum_{r=1}^{l}\binom{p}{r}\\
&\qquad\qquad\qquad
\cdot \left((i+j)-p(i+j)(i+j-1)H_{r-1}-\frac{2pij}{r}\right)
\end{align*}
where in the last step we applied \eqref{CC22}.
By \eqref{C3} and \eqref{C55}, we get
\begin{align*}
\sum_{l=1}^{p-1}\binom{p-1}{l}\sum_{r=1}^{l}\binom{p}{r}H_{r-1}
&\equiv_{p^2}-p\sum_{r=1}^{p-1}\frac{(-1)^rH_{r-1}}{r}\sum_{r=l}^{p-1}(-1)^l\\
&=-\frac{p}{2}\left(H_{p-1}(1,-1;-1)+H_{p-1}(1,1)\right)
\equiv_{p^2} -\frac{pq_p^2(2)}{2}.
\end{align*}
In a similar vain, \eqref{C2} and \eqref{C66} imply
$$\sum_{l=1}^{p-1}\binom{p-1}{l}\sum_{r=1}^{l}\binom{p}{r}\frac{1}{r}
\equiv_{p^2}-\frac{p}{2}\left(H_{p-1}(-2;-1)+H_{p-1}(2)\right)
\equiv_{p^2} 0.$$
Moreover, we have the identity
$$\sum_{l=1}^{p-1}\binom{p-1}{l}\sum_{r=1}^{l}\binom{p}{r}=2^{p-1}(2^{p-1}-1).$$
Finally, we obtain
$$\sum_{m=0}^{p-1}\binom{p-1+(i+j)p}{m+ip}
\equiv_{p^3}\binom{i+j}{i}\left(
2^{p-1}+(i+j)2^{p-1}(2^{p-1}-1)+p^2\binom{i+j}{2}q_p^2(2)
\right)$$
which is equivalent to \eqref{CC2}.
\end{proof}

\begin{lemma} Let $i, j$ be non-negative integers and $p>2$ a prime. Then, for $0\leq m\leq k<p$, we have
\begin{equation}\label{CC1}
\binom{k+(i+j)p}{m+ip}\equiv_{p^2}
\binom{i+j}{i}\binom{k}{m}(1+p((i+j)H_k-jH_{k-m}-iH_m)),
\end{equation}
and
\begin{equation}\label{CC11}
\sum_{m=0}^k\binom{k+(i+j)p}{m+ip}\equiv_{p^2}
2^k\binom{i+j}{i}\left(
1+p(i+j)\sum_{m=1}^{k}\frac{1}{m2^m}\right).
\end{equation}
\end{lemma}
\begin{proof} We begin with
\begin{align*}
\binom{k+(i+j)p}{m+ip}&=
\sum_{l=0}^{k}\binom{(i+j)p}{m-l+ip}\binom{k}{l}\\
&=\binom{i+j}{j}\binom{k}{m}+
\sum_{l=1}^{m}\binom{(i+j)p}{l+ip}\binom{k}{m-l}
+\sum_{l=1}^{k-m}\binom{(i+j)p}{l+jp}\binom{k}{m+l}.
\end{align*}
Then we apply \eqref{CC22} modulo $p^2$ so that
\begin{align*}
\binom{k+(i+j)p}{m+ip}&\equiv_{p^2}
\binom{i+j}{j}\left[\binom{k}{m}+
j\sum_{l=1}^{m}\binom{p}{l}\binom{k}{m-l}
+i\sum_{l=1}^{k-m}\binom{p}{l}\binom{k}{k-m-l}\right]\\
&=\binom{i+j}{j}\left[(1-j-i)\binom{k}{m}+
j\binom{p+k}{m}+i\binom{p+k}{k-m}\right]\\
&\equiv_{p^2}
\binom{i+j}{i}\binom{k}{m}(1+p((i+j)H_k-jH_{k-m}-iH_m)).
\end{align*}
From \eqref{CC1}, by summing over $m$ we obtain
$$\sum_{m=0}^k\binom{k+(i+j)p}{m+ip}\equiv_{p^2}
\binom{i+j}{i}\left(
2^k+p(i+j)\left(2^k H_k-\sum_{m=0}^k\binom{k}{m}H_m\right)
\right).$$
Congruence \eqref{CC11} follows after applying the identity from \eqref{H1}.
\end{proof}

\section{Proof of Theorem \ref{Main1}}

\noindent Let LHS and RHS be the left-hand side and the right-hand side of \eqref{SS}.
We have that
\begin{align*}
\mbox{LHS}&=
\sum_{m_1=0}^{rp-1}\sum_{m_2=0}^{sp-1} \binom{m_1+m_2}{m_1}
\sum_{m_3=0}^{tp-1}\binom{m_1+m_2+m_3}{m_3}\\
&=tp\sum_{m_1=0}^{rp-1}\sum_{m_2=0}^{sp-1} \binom{m_1+m_2}{m_1}
\binom{m_1+m_2+tp}{m_1+m_2}\frac{1}{m_1+m_2+1}\\
&=tp\sum_{m=0}^{rp-1}\sum_{k=m}^{m+sp-1} \binom{k}{m}
\binom{k+tp}{k}\frac{1}{k+1}\\
&=tp\sum_{i=0}^{r-1}\sum_{j=0}^{s-1}\sum_{m=0}^{p-1}\sum_{k=m}^{m+p-1} 
\binom{k+(i+j)p}{m+ip}\binom{k+(i+j+t)p}{k+(i+j)p}\frac{1}{k+(i+j)p+1}\\
&=\sum_{i=0}^{r-1}\sum_{j=0}^{s-1} (A_{ij}+B_{ij}+C_{ij})
\end{align*}
where
\begin{align*}
A_{ij}&:=tp\sum_{m=0}^{p-1}\sum_{k=m}^{p-1} 
\binom{k+(i+j)p}{m+ip}\binom{k+(i+j+t)p}{k+(i+j)p}\frac{1}{k+(i+j)p+1},\\
B_{ij}&:=\frac{tp}{(i+j+1)p+1}\binom{(i+j+t+1)p}{(i+j+1)p}
\sum_{m=1}^{p-1}\binom{(i+j+1)p}{m+ip},\\
C_{ij}&:=tp\sum_{m=0}^{p-1}\sum_{k=1}^{m-1} \binom{k+(i+j+1)p}{m+ip}
\binom{k+(i+j+t+1)p}{k+(i+j+1)p}\frac{1}{k+1+(i+j+1)p}.
\end{align*}
In a similar way
\begin{align*}
\mbox{RHS}
&=t\sum_{i=0}^{r-1}\sum_{j=0}^{s-1}
\binom{i+j}{i}\binom{i+j+t}{i+j}\frac{1}{i+j+1}.
\end{align*}
Wolstenholme's theorem $\binom{ap}{bp}\equiv_{p^3} \binom{a}{b}$ and
$\binom{n_1p+n_0}{k_1p+k_0}\equiv_{p} \binom{n_1}{k_1}\binom{n_0}{k_0}$ deliver
\begin{align*}
B_{ij}&\equiv_{p^3}\frac{t(i+j+1)p^2}{(i+j+1)p+1}\binom{i+j+t+1}{i+j+1}
\sum_{m=1}^{p-1}\binom{p-1+(i+j)p}{m-1+ip}\frac{1}{m+ip}\\
&\equiv_{p^3}tp^2(i+j+1)\binom{i+j+t+1}{i+j+1}\binom{i+j}{i}
\sum_{m=1}^{p-1}\binom{p-1}{m-1}\frac{1}{m}\\
&\equiv_{p^3}2tp^2(i+j+t+1)\binom{i+j+t}{i,j,t}q_p(2),
\end{align*}
because of the fact that
$$\sum_{m=1}^{p-1}\binom{p-1}{m-1}\frac{1}{m}=
\frac{1}{p}\sum_{m=1}^{p-1}\binom{p}{m}=\frac{2^p-2}{p}=2q_p(2).$$
Finally, we note that for $k<m<p$:
$$\binom{k+p}{m}
=\frac{1}{m!}\prod_{j=1}^{k}(j+p)\cdot p\cdot 
\prod_{j=1}^{m-(k+1)}(p-j)\equiv_{p^2} \frac{p(-1)^{m-(k+1)}}{(k+1)\binom{m}{k+1}},$$
and
\begin{align*}
\binom{k+(i+j+1)p}{m+ip}
&=\frac{(k+p+(i+j)p)\cdots(p+(i+j)p)}{(m+ip)\cdots(1+ip)}
\binom{k+p-m+(i+j)p}{ip}\\
&\equiv_{p^2}(i+j+1)\binom{k+p}{m}\binom{k-m+p+(i+j)p}{ip}\\
&\equiv_{p^2}(i+j+1)\frac{p(-1)^{m-(k+1)}}{(k+1)\binom{m}{k+1}}\binom{i+j}{i}.
\end{align*}
Furthermore,
$$\binom{k+(i+j+t+1)p}{k+(i+j+1)p}\equiv_{p} \binom{i+j+t+1}{i+j+1},$$ 
and thus we obtain
\begin{align*}
C_{ij}&\equiv_{p^3} tp^2(i+j+1)\binom{i+j+t+1}{i+j+1}\binom{i+j}{i}
\sum_{m=0}^{p-1}\sum_{k=1}^{m-1}\frac{(-1)^{m-(k+1)}}{(k+1)^2\binom{m}{k+1}}\\
&\equiv_{p^3}tp^2(i+j+t+1)\binom{i+j+t}{i,j,t}\sum_{k=1}^{p-2}\frac{1}{(k+1)^2}\sum_{m=k+1}^{p-1} \frac{(-1)^{m-(k+1)}}{\binom{m}{k+1}}\\
&\equiv_{p^3}tp^2(i+j+t+1)\binom{i+j+t}{i,j,t}\sum_{k=2}^{p-1}\frac{1}{k^2}\sum_{m=k}^{p-1} \frac{(-1)^{m-k}}{\binom{m}{k}}\\
&\equiv_{p^3} -2tp^2(i+j+t+1)\binom{i+j+t}{i,j,t}q_p(2)
\end{align*}
where in the last step we used \eqref{C7}.
Therefore $B_{ij}+C_{ij}\equiv_{p^3} 0$ and it suffices to show that
\begin{equation}\label{ff}
A_{ij}\equiv_{p^3}\frac{t}{i+j+1}\binom{i+j}{i}\binom{i+j+t}{i+j}.
\end{equation}
Now, by \eqref{CC2},
\begin{align*}
p\sum_{k=p-1}^{p-1} \sum_{m=0}^{k}\cdots
&=\frac{1}{i+j+t+1}\binom{(i+j+t+1)p}{(i+j+1)p}\sum_{m=0}^{k}\binom{p-1+(i+j)p}{m+ip}\\
&\equiv_{p^3}
\frac{1}{i+j+1}\binom{i+j+t}{i+j}\binom{i+j}{i}2^{(p-1)(i+j+1)}\\
&\equiv_{p^3}
\binom{i+j+t}{i,j,t}\left(\frac{1}{i+j+1}+pq_p(2)+\frac{(i+j)}{2}\,p^2q^2_p(2)\right)
\end{align*}
because
$$\frac{(2^{p-1})^{(i+j+1)}}{i+j+1}=\frac{(1+pq_p(2))^{i+j+1}}{i+j+1}
\equiv_{p^3} \frac{1}{i+j+1}+pq_p(2)+\frac{(i+j)}{2}\,p^2q^2_p(2).$$
Moreover by \eqref{CC1} and \eqref{CC11},
\begin{align*}
p\sum_{k=0}^{p-2} \sum_{m=0}^{k}\cdots
&\equiv_{p^3}
p\sum_{k=0}^{p-2}\left(\frac{1}{k+1}-\frac{p(i+j)}{(k+1)^2}\right)
\binom{i+j+t}{i+j}(1+p(i+j)H_k)\\
   &\qquad \cdot2^k\binom{i+j}{i}\left(
1+p(i+j)\sum_{m=1}^{k}\frac{1}{m2^m}\right)\\
&\equiv_{p^3}
\frac{p}{2}
\binom{i+j+t}{i,j,t}\Bigg(H_{p-1}(-1;2)-p(i+j)H_{p-1}(-2;2)\\
&\qquad\left. +p(i+j)H_{p-1}(1,-1;2) 
+p(i+j)\sum_{k=0}^{p-2}\frac{1}{k+1}\sum_{m=1}^{k}\frac{1}{m2^m}\right)\\
&\equiv_{p^3}
\binom{i+j+t}{i,j,t}\left(-pq_p(2)-\frac{(i+j)}{2}p^2q_p^2(2)\right),
\end{align*}   
where in the last step we used \eqref{C4}, \eqref{C10}, \eqref{C5}, and \eqref{B2}.
Hence
\begin{align*}
A_{ij}=tp\sum_{k=0}^{p-1} \sum_{m=0}^{k}\cdots
&\equiv_{p^3} \frac{t}{i+j+1}\binom{i+j+t}{i,j,t},
\end{align*}   
and the proof of \eqref{ff} follows. Therefore, Theorem \ref{Main1} has been validated.

\section{Supercongruence for sum of abelian squares}

\noindent
In this section, we consider a supercongruence for a triple sum of squared multinomial terms, which one can regard as a close companion to Theorem \ref{Main1}, and these appear in the following context (see \cite{R} for details). An \emph{abelian square} is a nonempty string of length $2n$ where the last $n$ symbols form a permutation of the first $n$ symbols. If $f_k(n)$ denotes the number of abelian squares of length $2n$ over an alphabet with $k$ letters, we have \cite{R}
$$f_k(n)=\sum_{n_1+\cdots+n_k=n}\binom{n_1+n_2+\cdots+n_k}{n_1,n_2,\dots,n_k}^2.$$

\noindent
We are ready to state and prove our second main result, namely the congruence that we promised for the abelian squares which is very much in the spirit of Section 4.

\begin{remark}
We remark that the congruence \eqref{News1} implies that the number of abelian squares over an alphabet with $3$ letters whose lengths are positive and less than $2p$, is divisible by $p^2$ (see entry $A174123$ in OEIS, Online Encyclopedia of Integer Sequences).
\end{remark}

\begin{theorem} 
Let $p>2$ be a prime, and let $r$, $s$, $t$ be any positive integers. Then
\begin{equation}\label{TT}
\sum_{m_1=0}^{rp-1}\sum_{m_2=0}^{sp-1}\sum_{m_3=0}^{tp-1}
\binom{m_1+m_2+m_3}{m_1,m_2,m_3}^2\equiv_{p^2} 
\sum_{m_1=0}^{r-1}\sum_{m_2=0}^{s-1}\sum_{m_3=0}^{t-1}
\binom{m_1+m_2+m_3}{m_1,m_2,m_3}^2.
\end{equation}
\end{theorem}
\begin{proof} We have that the LHS is
$$
\sum_{i_1=0}^{r-1}\sum_{i_2=0}^{s-1}\sum_{i_3=0}^{t-1}
\sum_{m_1=0}^{p-1}\sum_{m_2=0}^{p-1}\sum_{m_3=0}^{p-1}
\binom{m_1+m_2+m_3+(i_1+i_2+i_3)p}{m_3+i_3 p}^2
\binom{m_1+m_2+(i_1+i_2)p}{m_2+i_2 p}^2.
$$
If $m_1+m_2\geq p$ then $0\leq m_1+m_2-p<m_2<p$ and $p$ divides $\binom{m_1+m_2+(i_1+i_2)p}{m_2+i_2 p}$. Hence modulo $p^2$ we can assume that $m_1+m_2<p$. Using a similar argument for the other binomial coefficient, we can assume that $m_1+m_2+m_3<p$. With these assumptions, by using \eqref{CC1}, we obtain
that the LHS is congruent modulo $p^2$ to
\begin{align*}
&\sum_{i_1=0}^{r-1}\sum_{i_2=0}^{s-1}\sum_{i_3=0}^{t-1}
\binom{i_1+i_2+i_3}{i_1,i_2,i_3}^2\sum_{m_1+m_2+m_3<p}
\binom{m_1+m_2+m_3}{m_1,m_2,m_3}^2\\
&\qquad\cdot\left(1+2p((i_1+i_2+i_3)H_{m_1+m_2+m_3}
-i_1H_{m_1}-i_2H_{m_2}-i_3H_{m_3})\right).
\end{align*} 
So, it suffices to show that (this is the case $r=s=t=1$!)
\begin{align} \label{News1} 
\sum_{m_1+m_2+m_3<p}\binom{m_1+m_2+m_3}{m_1,m_2,m_3}^2\equiv_{p^2}1 \end{align}
and 
\begin{align} \label{News2}
\sum_{m_1+m_2+m_3<p}\binom{m_1+m_2+m_3}{m_1,m_2,m_3}^2
\left(H_{m_1+m_2+m_3}-H_{m_1}\right)\equiv_{p}0.
\end{align}
Before proving (\ref{News1}), we opt to rewrite the LHS and apply Vandermonde-Chu's identity followed by 
interchanging the order of summations so that
\begin{align*} \mbox{LHS}
=\sum_{a=0}^{p-1}\sum_{b=0}^a\sum_{c=0}^b\binom{b}c^2\binom{a}b^2
=\sum_{a=0}^{p-1}\sum_{b=0}^a\binom{2b}b\binom{a}b^2
=\sum_{b=0}^{p-1}\binom{2b}b\sum_{a=b}^{p-1}\binom{a}b^2. 
\end{align*}
Denote $p'=\frac{p-1}2$. Next, we break this up into two separate claims:
$$S_1:=\sum_{\substack{b=0 \\ b\neq p'}}^{p-1}\binom{2b}b\sum_{a=b}^{p-1}\binom{a}b^2\equiv_{p^2}0 \qquad
\text{and} \qquad S_2:=\binom{p-1}{p'}\sum_{a=p'}^{p-1}\binom{a}{p'}^2\equiv_{p^2}1.$$
As regards the first one,
\begin{align*}
S_1&=\sum_{b=0}^{p'-1}\binom{2b}b\sum_{a=b}^{p-1}\binom{a}b^2
+\sum_{b=p'+1}^{p-1}\binom{2b}b\sum_{a=b}^{p-1}\binom{a}b^2\\
&=
\sum_{b=0}^{p'-1}\left(\binom{2b}b\sum_{a=b}^{p-1}\binom{a}b^2
+\binom{2(p-1-b)}{p-1-b}\sum_{a=0}^{b}\binom{p-1-a}{b-a}^2\right)\\
&\equiv_{p^2}
\sum_{b=0}^{p'-1}\left(\binom{2b}{b}\sum_{a=b}^{p-1}\binom{a}b^2
+\binom{2(p-1-b)}{p-1-b}\binom{2b}{b}\right)\\
\end{align*}
because $p$ divides $\binom{2(p-1-b)}{p-1-b}$ for $b=0,\dots,p'-1$ and
$\binom{p-1-a}{b-a}\equiv_{p}(-1)^{b-a}\binom{b}{b-a}$.

\noindent Hence, it suffices to show that, for $b=0,\dots,p'-1$,
$$\sum_{a=b}^{p-1}\binom{a}{b}^2+\binom{2(p-1-b)}{p-1-b}\equiv_{p^2}0.$$
Replacing $c=p-1-b$, this claim tantamount: for $c=p'+1,\dots,p-1$
$$\sum_{a=0}^c\binom{p-1-a}{c-a}^2+\binom{2c}{c}\equiv_{p^2}0.$$
Since $p$ divides $\binom{2c}c$ for $c=p'+1,\dots,p-1$ and 
$$\binom{p-1-a}{c-a}\equiv_{p^2}(-1)^{c-a}\binom{c}{c-a}(1-p(H_c-H_a)),$$
it follows that
\begin{align*}
\sum_{a=0}^c\binom{p-1-a}{c-a}^2+\binom{2c}c
&\equiv_{p^2}
\sum_{a=0}^c\binom{c}a^2(1-2p(H_c-H_a))+\binom{2c}{c}\\
&=2\binom{2c}c(1-pH_c)+2p\sum_{a=0}^c\binom{c}a^2H_a\\
&=2\binom{2c}c\left(1-pH_c+p(2H_c-H_{2c})\right)\\
&\equiv_{p^2}2\binom{2c}c(1-pH_{2c})\equiv_{p^2} 2\binom{2c}c(1-1)=0,
\end{align*}
where we used the identity from \eqref{H2}, $p\binom{2c}cH_c\equiv_{p^2}0$ and $\binom{2c}cpH_{2c}\equiv_{p^2}1$.

\smallskip
\noindent
\noindent As regards $S_2$, we have that (see \cite[Lemma 2.5]{ZHS}), for $a=0,\dots,p'$,
$$\binom{p'+a}{p'}\equiv_{p^2} (-1)^a\binom{p'}{a}
\left(1+p\left(2H_{2a}-H_{a}\right)\right).$$
Hence
$$\binom{p-1}{p'}=\binom{p'+p'}{p'}\equiv_{p^2}
(-1)^{p'}\left(1+p\left(2H_{p-1}-H_{p'}\right)\right)
\equiv_{p^2}(-1)^{p'}(1-pH_{p'}),$$
and
\begin{align*}
S_2&=\binom{p-1}{p'}\sum_{a=0}^{p'}\binom{p'+a}{p'}^2\\
&\equiv_{p^2}(-1)^{p'}\sum_{a=0}^{p'}(-1)^a\binom{p'}{a}\binom{p'+a}{a}
\left(1+p(2H_{2a}-H_a-H_{p'})\right)
\end{align*}
After using identity \eqref{T1}, it suffices to show that 
\begin{equation}\label{CH1}
\sum_{a=1}^{p'}(-1)^{p'-a}\binom{p'}{a}\binom{p'+a}{a}(2H_{2a}-H_a-H_{p'})\equiv_{p}0.
\end{equation}
At this point, invoke the identities \eqref{T1}, \eqref{T2} and \eqref{T4}. Then, \eqref{CH1} becomes
$$6H_{p'}-2H_{\lfloor p'/2\rfloor} -2H_{p'}-H_{p'}=
3H_{p'}-2H_{\lfloor p'/2\rfloor}\equiv_{p}0$$
which holds because it is known (see \cite[congruences (41)-(44)]{L}) that $H_{p'}\equiv_{p} -2q_p(2)$ and $H_{\lfloor p'/2\rfloor}\equiv_{p} -3q_p(2)$.

\smallskip
\noindent
We now turn to the congruence (\ref{News2}), which can be reformulated as follows
\begin{align} \label{News2.1}
\sum_{a=0}^{p-1}\sum_{b=0}^a\sum_{c=0}^b\binom{b}c^2\binom{a}b^2H_a\equiv_p
\sum_{a=0}^{p-1}\sum_{b=0}^a\sum_{c=0}^b\binom{b}c^2\binom{a}b^2H_c. \end{align}
Applying the Vandermonde-Chu identity and swapping the order of summations, the left-hand side of (\ref{News2.1}) equals to
$$\text{LHS}=\sum_{a=0}^{p-1}\sum_{b=0}^a\binom{2b}b\binom{a}b^2H_a
=\sum_{b=0}^{p-1}\binom{2b}b\sum_{a=b}^{p-1}\binom{a}b^2H_a.$$
An analogous procedure allows to express the right-hand side of (\ref{News2.1}) as
$$\text{RHS}=\sum_{b=0}^{p-1}\binom{2b}b\sum_{a=b}^{p-1}\binom{a}b^2H_{a-b}.$$
Therefore, our task boils down to proving
$$\sum_{b=0}^{p-1}\binom{2b}b\sum_{a=b}^{p-1}\binom{a}b^2H_a\equiv_p
\sum_{b=0}^{p-1}\binom{2b}b\sum_{a=b}^{p-1}\binom{a}b^2H_{a-b}.$$
In fact, since $p$ divides $\binom{2b}b$ for $p'<b<p$, it suffices to show, for $0\leq b\leq p'$, that
\begin{align} \label{News2.2}
\sum_{a=b}^{p-1}\binom{a}b^2(H_a-H_{a-b})\equiv_p 0.\end{align}
The case $b=0$ is trivial. Note that if $p-1$ does not divide $i$ then $\sum_{a=1}^{p-1}a^i\equiv_p0$ (see, for instance \cite[equation (16)]{L}); Or, more directly, since $\sum_{a=1}^{p-1}a\equiv_p0$ and because the map $a\mapsto a^i$ is an isomorphism of the group $\mathbb{F}_p^*$, we have $\sum_{a=1}^{p-1}a^i\equiv_p0$. On this basis and starting the summation (\emph{harmlessly}) at $a=1$, we obtain 
\begin{align*} \sum_{a=b}^{p'}\binom{a}b^2(H_a-H_{a-b})
&=\sum_{a=b}^{p-1}\binom{a}b^2\sum_{j=0}^{b-1}\frac{1}{a-j}
=\sum_{a=1}^{p-1}\binom{a}b^2\sum_{j=0,j\not=a}^{b-1}\frac{1}{a-j}
\\
&
=\frac1{b!^2}\sum_{a=1}^{p-1}\sum_{i=1}^{2b-1}\alpha_i\,a^i
=\sum_{i=1}^{2b-1}\frac{\alpha_i}{b!^2}\sum_{a=1}^{p-1}a^i\equiv_p 0,
\end{align*}
for some integers $\alpha_i$. Since $1\leq i\leq 2p'-1=p-2$, the proof follows.
\end{proof}

\section{Some further remarks}

\noindent
Let us consider the  {\sl double-sum} counterparts of those sums from the last two sections.
First of all, it is evident that
$$\sum_{m_1=0}^{rp-1}\sum_{m_2=0}^{sp-1} \binom{m_1+m_2}{m_1}=\binom{(s+r)p}{rp}-1
\equiv_{p^3}\binom{s+r}{r}-1=\sum_{m_1=0}^{r-1}\sum_{m_2=0}^{s-1}
\binom{m_1+m_2}{m_1}.$$
Moreover, the Supercongruence 5' in \cite{AZ} says that
$$\sum_{m_1=0}^{rp-1}\sum_{m_2=0}^{sp-1}
\binom{m_1+m_2}{m_1}^2\equiv_{p^2}
\leg{p}{3}\sum_{m_1=0}^{r-1}\sum_{m_2=0}^{s-1}\binom{m_1+m_2}{m_1}^2.$$
As a matter of fact, as before, the LHS is
\begin{align*}
&\sum_{i_1=0}^{r-1}\sum_{i_2=0}^{s-1}
\binom{i_1+i_2}{i_1}^2\sum_{m_1+m_2<p}
\binom{m_1+m_2}{m_1}^2\cdot\left(1+2p((i_1+i_2)H_{m_1+m_2}-i_1H_{m_1}-i_2H_{m_2})\right).
\end{align*} 
Now, by Vandermonde's convolution,  
\begin{align*} 
\sum_{m_1+m_2<p}\binom{m_1+m_2}{m_1}^2\equiv_{p^2}
\sum_{k=0}^{p-1}\sum_{m_1=0}^{k} \binom{k}{m_1}^2=
\sum_{k=0}^{p-1}\binom{2k}{k}\equiv_{p^2} \leg{p}{3}
 \end{align*}
where in the last step we used \cite[congruence (1.9)]{ST}. Here, $\leg{p}{3}$ is the Legendre symbol.
So, it suffices to show that
\begin{align} 
\sum_{m_1+m_2<p}\binom{m_1+m_2}{m_1}^2
\left(H_{m_1+m_2}-H_{m_1}\right)\equiv_{p}0
\end{align}
which can be proved as before.
By \cite[Proposition 7']{AZ}, if $p>3$ is a prime 
and $r_1, \dots, r_n$ are positive integers, then
$$
\sum_{m_1=0}^{r_1p-1} \dots \sum_{m_n=0}^{r_np-1} 
\binom{m_1+\dots +m_n}{m_1, \dots , m_n}  
\equiv_p  
\sum_{m_1=0}^{r_1-1} \dots \sum_{m_n=0}^{r_n-1}  
\binom{m_1+\dots + m_n}{m_1, \dots , m_n} .
$$
What can be said of 
$$\sum_{m_1=0}^{p-1} \dots \sum_{m_n=0}^{p-1} 
\binom{m_1+\dots +m_n}{m_1, \dots , m_n}^2  
\equiv_p  \sum_{m_1+\dots+m_n<p}\binom{m_1+\dots +m_n}{m_1, \dots , m_n}^2
\equiv_p?$$

\smallskip
\noindent
Apagodu and Zeilberger \cite{AZ} succeeded to prove Theorem \ref{Main1}, by the \emph{constant term} method, modulo $p$, although our method yields modulo $p^3$. Below, an intermediate congruence (modulo $p^2$) is established by allowing the two techniques to work in tandem. For notational simplicity, adopt the cyclic notation $\prod_{\substack{cyc}}U(x,y,z)=U(x,y,z)U(z,x,y)U(y,z,x)$. Leaving out the initial steps, for which the reader is referred to \cite{AZ}, we commence with:
\begin{align*}
\sum_{m_1,m_2,m_3=0}^{p-1}\binom{m_1+m_2+m_3}{m_1,m_2,m_3}
&=CT\frac1{(xyz)^{p-1}}\prod_{cyc} \frac{(x+y+z)^p-x^p}{y+z} \\
&=CT\frac1{(xyz)^{p-1}}\prod_{cyc}\left((y+z)^{p-1}+\sum_{i=1}^{p-1}\binom{p}ix^{p-1}(y+z)^{i-1}\right).
\end{align*}
We proceed with this equation by identifying two separate contributors. For the first piece, use $\binom{p-1}j\equiv_{p^2}(-1)^j\left[1-pH_j\right]$ and $H_{p'}\equiv_p-2q_p(2)$ (see \cite[(41)]{L}) so that
\begin{align*} CT\frac1{(xyz)^{p-1}}\prod_{cyc}(y+z)^{p-1}
&=CT\frac1{(xyz)^{p-1}}\prod_{cyc}\sum_{i=0}^{p-1}\binom{p-1}iy^iz^{p-1-i}=\sum_{j=0}^{p-1}\binom{p-1}j^3 \\
&\equiv_{p^2}\sum_{j=0}^{p-1}(-1)^j\left[1-3pH_j\right] 
=1-\frac{3p}2H_{p'}\equiv_{p^2}1+3p\,q_p(2). \end{align*}
For the second part, apply $\binom{p}k\equiv_{p^2}(-1)^{k-1}\frac{p}k, \,\binom{p-1}j\equiv_p(-1)^j$ and \eqref{C4} to obtain
\begin{align*} 
&CT\frac1{(xyz)^{p-1}}
\prod_{cyc}\left((y+z)^{p-1}(x+z)^{p-1}\sum_{k=1}^{p-1}\binom{p}kz^{p-k}(x+y)^{k-1}\right) \\
=&CT\frac1{(xyz)^{p-1}}\prod_{cyc}
\left(\sum_{i,j=0}^{p-1}\binom{p-1}iy^{p-1-i}z^i\binom{p-1}jx^{p-1-j}z^j
\sum_{k=1}^{p-1}\binom{p}kz^{p-k}(x+y)^{k-1}\right) \\
=&3\sum_{\substack{i,j=0 \\  i+j<p-1}}^{p-1}\binom{p-1}i\binom{p-1}j\binom{p}{i+j+1}\binom{i+j}i \\
\equiv&_{p^2}3p\sum_{\substack{i,j=0 \\  i+j<p-1}}^{p-1}\frac1{i+j+1}\binom{i+j}i 
=3p\sum_{k=0}^{p-2}\frac1{k+1}\sum_{j=0}^k\binom{k}j \\
=&3p\sum_{k=0}^{p-2}\frac{2^k}{k+1}\equiv_{p^2}-3p\,q_p(2). \end{align*}
Combining the two pieces, we arrive at $\sum_{m_1,m_2,m_3=0}^{p-1}\binom{m_1+m_2+m_3}{m_1,m_2,m_3}\equiv_{p^2}1$.

\medskip


\begin{thebibliography}{99}

\bibitem{AZ} M. Apagodu and D. Zeilberger,
\emph{Using the ``Freshman's Dream'' to Prove Combinatorial Congruences},
arXiv:1606.03351 (june 2016).

\bibitem{B} F. Beukers, 
\emph{Another congruence for the A\'pery numbers},
J. Number Theory \textbf{25} (1987), no. 2, 201-210.

\bibitem{CCS} H.-H. Chan, S. Cooper, and F. Sica,
\emph{Congruences  satisfied  by Ap\'ery-like numbers}, 
Int. J. Number Theory \textbf{6} (2010), no. 1, 89-97.



\bibitem{G} H. W. Gould,
\emph{Combinatorial Identities}, Morgantown W. Va (1972).

\bibitem{H} L. van Hamme,
\emph{Some conjectures concerning partial sums of generalized hypergeometric series,
$p$-adic functional analysis} (Nijmegen, 1996), 223-236, Lecture Notes in Pure and Appl. Math., 192, 
Dekker, New York, 1997.

\bibitem{L} E. Lehmer, 
\emph{On Congruences Involving Bernoulli Numbers and the Quotients of Fermat and Wilson}, 
Ann. of Math., \textbf{39} (1938), 350-360.

\bibitem{PR} H. Prodinger,
\emph{Identities involving harmonic numbers that are of interest for physicists},
Util. Math., \textbf{83} (2010), 291-299.

\bibitem{S} P. Paule and C. Schneider, 
\emph{Computer proofs of a new family of harmonic number identities}, 
Adv. Appl. Math., \textbf{31} (2003), 359-378.

\bibitem{R} L. B. Richmond and J. Shallit,
\emph{Counting Abelian squares},
Elec. J. Comb. \textbf{16} (2009), \#R72.

\bibitem{ZHS} Z.-H. Sun, 
\emph{Congruences concerning Legendre polynomials}, 
Proc. Amer. Math. Soc., \textbf{139} (2011), 1915-1929.

\bibitem{ZWS1} Z.-W. Sun,
\emph{On congruences related to central binomial coefficients}, 
J. Number Theory \textbf{131} (2011), 2219-2238.

\bibitem{ZWS2} Z.-W. Sun,
\emph{Super congruences and Euler numbers}, 
Sci. China Math. \textbf{54} (2011), 2509-2535.

\bibitem{ST} Z.-W. Sun and R. Tauraso, 
\emph{On some new congruences for binomial coefficients}, 
Int. J. Number Theory, \textbf{7} (2011), 645-662.


\bibitem{TZ} R. Tauraso and J. Zhao, 
\emph{Congruences of alternating multiple harmonic sums}, 
J. Comb. Number Theory, \textbf{2} (2010), 129-159.

\bibitem{RV} F. Rodriguez-Villegas,
\emph{Hypergeometric families of Calabi-Yau manifolds. Calabi-Yau varieties and mirror symmetry},
(Toronto, ON, 2001), 223–243, Fields Inst. Commun., 38, Amer. Math. Soc., Providence, RI, 2003.

\bibitem{Z} W. Zudilin,
\emph{Ramanujan-type supercongruences}, J. Number Theory, \textbf{129} (2009), no. 8, 1848-1857.

\end{thebibliography}
\end{document}